\documentclass[11pt]{article}
\usepackage{amsmath}
\usepackage{amsthm}
\usepackage{latexsym}
\usepackage{hyperref}

\newtheorem*{urem}{Remark}
\newtheorem{theorem}{Theorem}[section]
\newtheorem{rem}[theorem]{Remark}
\newtheorem{lem}[theorem]{Lemma}
\newtheorem{propos}[theorem]{Proposition}
\newtheorem{defin}[theorem]{Definition}

\begin{document}

\title{On the Union of Arithmetic Progressions}
\author{
Shoni Gilboa\thanks{Mathematics Dept., The Open University of Israel, Raanana 43107 , Israel.  \texttt{tipshoni@gmail.com}.}
\and
Rom Pinchasi\thanks{
Mathematics Dept.,
      Technion---Israel Institute of Technology,
      Haifa 32000, Israel.
      \texttt{room@math.technion.ac.il}. Supported by ISF grant (grant No. 1357/12) and by BSF grant (grant No. 2008290).
} 
}

\maketitle

\begin{abstract}
We show that for every $\varepsilon>0$ there is an absolute constant 
$c(\varepsilon)>0$ such that the following is true.
The union of any $n$ arithmetic progressions, each of 
length $n$, with pairwise distinct differences must consist of at least 
$c(\varepsilon)n^{2-\varepsilon}$ elements.
We observe, 
by construction, that one can find $n$ arithmetic progressions, each 
of length $n$, with pairwise distinct differences such that the cardinality of 
their union is $o(n^2)$.
We refer also to the non-symmetric case of $n$ arithmetic
progressions, each of length $\ell$, for various regimes of $n$ and $\ell$.
\end{abstract}

\section{Introduction}

For integers $n>1$ and $\ell>1$ let $u_{\ell}(n)$ 
be the minimum possible cardinality of a union of $n$ arithmetic progressions,
each of length $\ell$, with pairwise distinct differences. 

Clearly, $u_{\ell}(n)\leq n\,\ell$, but this inequality is not tight in general, not even up to a multiplicative absolute constant. 

For small values of $\ell$ it is not hard to see, for instance, that 
$u_2(n)=\lceil\frac{1}{2}+\sqrt{2n+\frac{1}{4}}\rceil$.
Ruzsa proved (in \cite{Rusza}) that there is an absolute constant $c>0$, such that for any sufficiently large integer $m$ there is a set $A$ of $m$ integers such that $|A+A|\leq m^{2-c}$ but $|A-A|\geq m^2-m^{2-c}$. 
For any positive $d\in A-A$ choose $a_d,b_d\in A$ such that $b_d-a_d=d$. Now $\left(2a_d,a_d+b_d,2b_d\right)_{d\in A-A}$ is a set of $\frac{|A-A|-1}{2}$ 
arithmetic progressions of length $3$ with pairwise distinct differences, 
and their union is contained in $(A+A)\cup 2A$. We thus obtain at least 
$\frac{m^2-m^{2-c}-1}{2}$ arithmetic progressions of length $3$, with 
pairwise distinct differences, such that the cardinality of their union is at most $m^{2-c}+m$. It follows that $u_3(n)=O(n^{1-\frac{c}{2}})$. 

The trivial upper bound $u_{\ell}(n)\leq n\,\ell$ is far from being tight for large values of $\ell$ as well. In particular, consider the symmetric case where $\ell=n$. It turns out that
$u_n(n)=o(n^2)$.
To see this take (perhaps the most natural choice of) arithmetic progressions:
$A_j=\{i\,j \mid i \in [n]\}$ for $j=1, \ldots, n$.
Clearly, the union $\bigcup_{j=1}^{n}A_j$ is precisely the set
$\{i\,j \mid i,j \in [n]\}$.  
It was shown already by Erd\H os in 1955 (\cite{Erdos}) that 
$$\lvert\{i\, j\mid i,j\in[n]\}\rvert=o\left(n^2/(\log n)^{\alpha}\right)$$
for some $\alpha>0$. The exact assymptotics
$$\left\lvert\left\{i\, j\mid i,j\in[n]\right\}\right\rvert\sim\frac{n^2}{(\log n)^{1-\frac{1+\log\log 2}{\log 2}}(\log\log n)^{\frac{3}{2}}}$$ 
 was given in 2008 by Ford (\cite{Ford}).
Consequently, we obtain the desired improved upper bound for $u_{n}(n)$.
\vspace{5pt}

In this paper we show that $u_{\ell}(n)$ cannot be much smaller than 
$n\,\ell$, provided $\ell$ is not much smaller than $n$, as captured in the following theorem, giving a lower bound for $u_{\ell}(n)$ for smaller values of $\ell$ as well.

\begin{theorem}\label{thm:gen-prog} For every $\varepsilon>0$ there is a positive constant $c_1(\varepsilon)$, depending only on $\varepsilon$, such that for any positive integers $n$ and $\ell$  
\begin{equation*}\label{eq:gen-prog}u_{\ell}(n)\geq\begin{cases} c_1(\varepsilon)\,n^{\frac{1}{2}-\varepsilon}\,\ell&\text{for }\;\ell\leq n^{\frac{1}{2}-\varepsilon}\\c_1(\varepsilon)\,\ell^2&\text{for }\; n^{\frac{1}{2}-\varepsilon}\leq\ell\leq n^{1-\varepsilon}\\ c_1(\varepsilon)\, n^{1-\varepsilon}\,\ell&\text{for }\; n^{1-\varepsilon}\leq\ell.\end{cases}\end{equation*}
\end{theorem}

In the proof of Theorem \ref{thm:gen-prog} 
we study and use upper bounds for the following two functions, 
that are of independent interest.
$$
f_d(m,n)=\max_{\substack{A,B\subset(0,\infty)\\|A|\leq m~,~|B|\leq n}}
\left\lvert\left\{(a,b)\in A\times B\mid\frac{a}{b}\in[d]\right\}\right\rvert,
$$
$$
g_d(n)=\max_{\substack{B\subset(0,\infty)\\|B|\leq n}}
\left\lvert\left\{(b_1,b_2)\in B^2\mid b_1<b_2,~\exists p,q\in[d]:\frac{b_1}{b_2}=\frac{p}{q}\right\}\right\rvert.
$$
\medskip

The paper is organized as follows. In Sections \ref{sec:g} we provide an upper
bound for the function $g_{d}$ above. Using this bound, we provide 
an upper bound for the function $f_{d}$ in Section 
\ref{sec:f}. 
Theorem \ref{thm:gen-prog} is proved in 
Section \ref{sec:prog}. Section \ref{sec:con} contains 
one of many possible number theory applications to the upper bounds for 
$f_{d}$ and $g_{d}$.

\section{Rational quotients with bounded numerator and denominator}\label{sec:g}
For positive integer $d$ define

$$
R_d=\left\{\frac{k}{\ell} \mid k,\ell\in[d]\right\}.
$$

\begin{defin}\label{definition:g}
For a positive integer $d$ and 
a finite set $B$ of positive real numbers define 
$$ \mathcal{G}_d(B)=\left\{\{b_1,b_2\} \subseteq B\mid 1\neq\frac{b_1}{b_2}\in R_d\right\}.$$

For positive integers $n$ and $d$ define 
$$
g_d(n)=\max_{\substack{B\subset(0,\infty)\\|B|\leq n}}
|\mathcal{G}_{d}(B)|.
$$
\end{defin} 

Clearly,
\begin{equation}\label{eq:g_easy_bound} g_d(n)\leq(n-1)|R_d|<n\, d^2.\end{equation}

This bound is useful when $d$ is small. For large values of $d$ we have 
the following improved upper bound:

\begin{propos}\label{g} 
For any positive integer $k$, there is a positive constant $c_2(k)$, depending only on $k$, such that for any positive integer $n$ and any 
integer $d>c_2(k)$
\begin{equation}\label{eq:g}g_d(n)<(200k+1)n^{1+\frac{1}{k}}d^{1-\frac{1}{2k}}.\end{equation}
\end{propos}

The proof of Proposition \ref{g} will follow by comparing upper and lower bounds
for the cardinality of the set 
$$
\mathcal{C}_{d,2k}(B)=\left\{(b_1,b_2,\ldots,b_{2k})\in B^{2k} \mid 
\frac{b_1}{b_{2}},\frac{b_2}{b_{3}},\ldots,\frac{b_{2k-1}}{b_{2k}},\frac{b_{2k}}{b_1}\in R_d,\,\forall 1\leq i<j\leq 2k:b_i\neq b_j\right\},
$$
in terms of $|\mathcal{G}_{d}(B)|$,
where $B$ is a finite set of positive real numbers.

We start with bounding the cardinality of $\mathcal{C}_{d,2k}(B)$ from below.
We first get a basic lower bound for $|\mathcal{C}_{d,2k}(B)|$ using the Bondy-Simonovits Theorem (\cite{BondySimonovits}), which states that a graph with $n$ vertices and no simple 
cycles of length $2k$ has no more than $100 k\cdot n^{1+\frac{1}{k}}$ edges.
Later, we enhance this basic lower bound for $|\mathcal{C}_{d,2k}(B)|$
in the case where $|\mathcal{G}_{d}(B)|$ is large. 

\begin{lem}\label{basic-lower}
For any positive integers $k$ and $d$, and for any finite set $B$ of positive real numbers,
$$
\frac{1}{4k}\left\lvert\mathcal{C}_{d,2k}(B)\right\rvert
\geq
\left\lvert\mathcal{G}_{d}(B)\right\rvert-100k\left\lvert B\right\rvert^{1+\frac{1}{k}}.
$$
\end{lem}
\begin{proof} Form a graph on the vertex set $B$, by connecting two distinct vertices $b_1,b_2\in B$ if and only if 
$\frac{b_1}{b_2}\in R_d$. This graph obviously has $|B|$ vertices, 
$|\mathcal{G}_{d}(B)|$ edges, and at most $\frac{1}{2\cdot 2k}\left\lvert \mathcal{C}_{d,2k}(B)\right\rvert$ 
simple cycles of length $2k$. 
Now remove an edge from every simple cycle of length $2k$ in this graph. 
We get a graph with $|B|$ vertices and at least 
$$|\mathcal{G}_{d}(B)|-\frac{1}{4k}\left\lvert \mathcal{C}_{d,2k}(B)\right\rvert$$ edges.
The resulting graph has no simple cycle of length $2k$. The result now follows 
directly from the Bondy-Simonovits Theorem stated above.
\end{proof}

Next, we use a standard probabilistic argument to enhance the lower bound of 
Lemma \ref{basic-lower}.

\begin{lem}\label{adv-lower}
For any positive integers $k$ and $d$, and for any finite set $B$ of positive real numbers
such that $|B|>d^{1+\frac{1}{2(k-1)}}$, we have the following inequality:
\begin{equation*}\left\lvert\mathcal{G}_d(B)\right\rvert<\frac{1}{4k}\left\lvert\mathcal{C}_{d,2k}(B)\right\rvert d^{-(2k-1)}+200k\, n^{1+\frac{1}{k}}d^{1-\frac{1}{2k}}.\end{equation*}
\end{lem}
\begin{proof} Let $p:=d^{-1-\frac{1}{2(k-1)}}$,  
and let $B_p$ be a random subset of $B$ obtained by choosing each element independently with probability $p$.
By Lemma \ref{basic-lower},
$$\frac{1}{4k}\left\lvert \mathcal{C}_{d,2k}(B_p)\right\rvert\geq
|\mathcal{G}_{d}(B_{p})|-100k\,\lvert B_p\rvert^{1+\frac{1}{k}}.$$

Taking expectations, we get
\begin{equation}\label{E}\frac{1}{4k}E\left\lvert \mathcal{C}_{d,2k}(B_p)\right\rvert\geq E|\mathcal{G}_{d}(B_p)|-100k\, E\left(\lvert B_p\rvert^{1+\frac{1}{k}}\right).\end{equation}
Notice that from the linearity of expectation we have: 
\begin{equation}\label{left}E\left\lvert \mathcal{C}_{d,2k}(B_p)\right\rvert=\left\lvert \mathcal{C}_{d,2k}(B)\right\rvert p^{2k}
\end{equation}
and
\begin{equation}\label{first-term}
E|\mathcal{G}_{d}(B_p)|=|\mathcal{G}_{d}(B)|p^2.
\end{equation}
As for $E\left(\lvert B_p\rvert^{1+\frac{1}{k}}\right)$, note that $E|B_p|=|B| p$ and $V|B_p|=|B| p(1-p)$. Therefore, since $1<|B| p$,
$$
E\left(|B_p|^2\right)=V|B_p|+\left(E|B_p|\right)^2=
|B| p(1-p)+(|B| p)^2<2|B|^2p^2.
$$
Now, by Holder's Inequality,
\begin{equation}\label{second-term}E\left(\lvert B_p\rvert^{1+\frac{1}{k}}\right)\leq\left(E\left(\lvert B_p\rvert^2\right)\right)^{\frac{1}{2}+\frac{1}{2k}}<2 |B|^{1+\frac{1}{k}}p^{1+\frac{1}{k}}.\end{equation}
Plugging \eqref{left}, \eqref{first-term} and \eqref{second-term} in \eqref{E} we get
$$
\frac{1}{4k}\left\lvert \mathcal{C}_{d,2k}(B)\right\rvert p^{2k}>
|\mathcal{G}_{d}(B)|p^2-200k|B|^{1+\frac{1}{k}}p^{1+\frac{1}{k}},
$$
hence
\begin{equation*}
\left\lvert\mathcal{G}_d(B)\right\rvert<\frac{1}{4k}\left\lvert \mathcal{C}_{d,2k}(B)\right\rvert p^{2k-2}+200k\, |B|^{1+\frac{1}{k}}p^{-1+\frac{1}{k}}=\frac{1}{4k}\left\lvert 
\mathcal{C}_{d,2k}(B)\right\rvert d^{-(2k-1)}+200k\, n^{1+\frac{1}{k}}d^{1-\frac{1}{2k}}.\qedhere\end{equation*}
\end{proof}

We now approach the task of bounding $|\mathcal{C}_{d,2k}(B)|$ from above. 
We start with the following well known number-theoretic bound on the number of 
divisors $d(m)$ of an integer $m$.
\begin{lem}\label{divisors} For any $\delta>0$ there is a positive constant $c_3(\delta)$ depending only on $\delta$, such that for any positive integer $m$,
$$d(m)<c_3(\delta)m^{\delta}.$$ 
\end{lem}
\begin{proof} Let $m=\prod_{i=1}^kp_i^{r_i}$ be the prime factorization of $m$. $m$ has $d(m)=\prod_{i=1}^k(1+r_i)$ divisors. For any $1\leq i\leq k$, 
$$\left(p_i^{r_i}\right)^{\delta}=e^{\delta r_i\ln p_i}> 1+\delta r_i\ln p_i.$$
Therefore,
$$\frac{d(m)}{m^{\delta}}=\prod_{i=1}^k\frac{1+r_i}{\left(p_i^{r_i}\right)^{\delta}}<\prod_{i=1}^k\frac{1+r_i}{1+\delta r_i\ln p_i}\leq\prod_{i=1}^k\frac{1}{\min\{1,\delta\ln p_i\}}=\prod_{\substack{1\leq i\leq k\\\ln p_i\leq 1/\delta}}\frac{1}{\delta\ln p_i}\leq\prod_{\substack{p \text{ prime}\\p\leq e^{1/\delta}}}\frac{1}{\delta\ln p}.$$
Hence, $d(m)<c_3(\delta)m^{\delta}$, 
where $c_3(\delta):=\displaystyle\prod_{\substack{p \text{ prime}\\p\leq e^{1/\delta}}}\frac{1}{\delta\ln p}$.
\end{proof}

\begin{lem}\label{cor-upper}  For any positive integer $k$ there is a positive constant $c_4(k)$, depending only on $k$, such that for any positive integer $d$ and any finite set $B$ of positive real numbers we have
$$\lvert \mathcal{C}_{d,2k}(B)\rvert<c_4(k)\,|B|\,d^{\,2k+\frac{1}{4k}}.$$
\end{lem}
\begin{proof}
We notice that
\begin{eqnarray}
\lvert\mathcal{C}_{d,2k}(B)\rvert& \leq &
\left\lvert\left\{(b_1,b_2,\ldots,b_{2k})\in B^{2k} \mid 
\forall 1\leq i\leq {2k}-1: \frac{b_i}{b_{i+1}}\in R_d, 
~\frac{b_{2k}}{b_1}\in R_d\right\}\right\rvert\leq\nonumber\\
&\leq &|B|\cdot\left\lvert\left\{(r_1,r_2,\ldots,r_{2k}) \mid 
\forall 1\leq i\leq {2k}: 
r_i\in R_d, ~r_1r_2\cdots r_{2k}=1\right\}\right\rvert\leq\nonumber\\
&\leq&  |B|\cdot\left\lvert\left\{((p_1,p_2,\ldots,p_{2k}),
(q_1,q_2,\ldots,q_{2k}))\in\left([d]^{2k}\right)^2 \mid 
p_1p_2\cdots p_{2k}=q_1q_2\cdots q_{2k}\right\}\right\rvert\leq\nonumber\\
&\leq& |B|\cdot\sum_{m=1}^{d^{2k}}\left\lvert\left\{(p_1,p_2,\ldots,p_{2k})\in [d]^{2k} \mid p_1p_2\cdots p_{2k}=m\right\}\right\rvert^2\leq |B|\cdot\sum_{m=1}^{d^{2k}}{d(m)}^{2k}.\nonumber
\end{eqnarray}
By Lemma \ref{divisors}, $d(m)<c_3(1/16k^3)m^{1/16k^3}$ for any $m$, and we get
\begin{equation*}\lvert\mathcal{C}_{d,2k}(B)\rvert<|B|\cdot\sum_{m=1}^{d^{2k}}\left(c_3(1/16k^3)m^{1/16k^3}\right)^{2k}\leq\left(c_3(1/16k^3)\right)^{2k}|B|\,d^{2k+\frac{1}{4k}}.
\end{equation*}
This completes the proof with $c_4(k):=\left(c_3(1/16k^3)\right)^{2k}$.
\end{proof}

We are now prepared for proving Proposition \ref{g}.
\begin{proof}[Proof of Proposition \ref{g}]

If $n\leq d^{1+\frac{1}{2(k-1)}}$, then \eqref{eq:g} holds because
$$
g_d(n)\leq\binom{n}{2}<n^2=n^{1+\frac{1}{k}}n^{1-\frac{1}{k}}\leq n^{1+\frac{1}{k}}\left(d^{1+\frac{1}{2(k-1)}}\right)^{1-\frac{1}{k}}=n^{1+\frac{1}{k}}d^{1-\frac{1}{2k}}.
$$

We therefore assume $n>d^{1+\frac{1}{2(k-1)}}$. Let $B$ be a set of $n$ positive real numbers. By Lemma \ref{adv-lower},
\begin{equation}\label{1}\left\lvert\mathcal{G}_d(B)\right\rvert 
<200k\, n^{1+\frac{1}{k}}d^{1-\frac{1}{2k}}+\frac{1}{4k}\left\lvert 
\mathcal{C}_{d,2k}(B)\right\rvert d^{-(2k-1)}.
\end{equation}
By Lemma \ref{cor-upper},
\begin{equation*}
\lvert \mathcal{C}_{d,2k}(B)\rvert<c_4(k) n\, d^{\,2k+\frac{1}{4k}}.
\end{equation*}
Hence, for $d\geq c_2(k):=\left(c_4(k)/4k\right)^{4k}$,
\begin{equation}\label{2}
\lvert \mathcal{C}_{d,2k}(B)\rvert<4k\, n\, d^{\,2k+\frac{1}{2k}}.
\end{equation}
Plugging \eqref{2} in \eqref{1} and using our assumption that 
$n>d^{1+\frac{1}{2(k-1)}}\geq d$, we get that for $d\geq c_2(k)$,
\begin{equation*}
\left\lvert\mathcal{G}_d(B)\right\rvert <
200k\, n^{1+\frac{1}{k}}d^{1-\frac{1}{2k}}+n\, d^{1+\frac{1}{2k}}<  (200k+1)n^{1+\frac{1}{k}} d^{1-\frac{1}{2k}}.\qedhere\end{equation*}
\end{proof}

\section{Bounded integer quotients}\label{sec:f}
\begin{defin} For positive integers $m,n,$ and $d$ define
$$f_d(m,n)=\max_{\substack{A,B\subset(0,\infty)\\|A|\leq m~,~|B|\leq n}}
\lvert\{(a,b)\in A\times B \mid \frac{a}{b}\in[d]\}\rvert.
$$
\end{defin} 
\begin{rem}\label{remark} It is an amusing exercise to see that 
$f_d(m,n)=f_d(n,m)$. Therefore, we may assume, if needed, with no loss of 
generality that $m\leq n$, or that $m\geq n$.
\end{rem}

\begin{propos}\label{f} For any $\varepsilon>0$ there is a positive constant 
$c_6(\varepsilon)$, depending only on $\varepsilon$, such that for any positive integers $m,n$, and $d$
$$
f_d(m,n)< c_5(\varepsilon)\left(\min\{\sqrt{n},\sqrt{m}\}\right)^{\varepsilon}\sqrt{m\, n\, d}.
$$
\end{propos}
\begin{proof} With no loss of generality (see Remark \ref{remark}) assume that $n\leq m$. We may also assume that $m< n\,d$ , because if $m\geq n\, d$, 
then $f_d(m,n)\leq n\, d=\sqrt{(n\, d) n\, d}\leq\sqrt{m\, n\, d}$ (see Proposition \ref{three} for further discussion).

Let $A$ and $B$ be finite sets of positive real numbers such that $|A|\leq m,|B|\leq n$. 
The proposition will follow by comparing lower and upper bounds for the cardinality of the set
$$
W=\{(a,b_1,b_2) \in A\times B^2 \mid 
\frac{a}{b_1},\frac{a}{b_2}\in [d], ~~b_1<b_2\}.
$$

We first establish an upper bound for $|W|$. For convenience define
$$
S_d=\{(p,q) \mid p,q\in[d],~~p<q,~~\gcd(p,q)=1\}.
$$ 

We have:

\begin{eqnarray}
|W| &= &\left\lvert\left\{(a,b_1,b_2)\in A\times B^2 \mid 
\frac{a}{b_1},\frac{a}{b_2}\in [d],~~b_1<b_2\right\}\right\rvert=\nonumber\\
 &=& \left\lvert\left\{(b_1,b_2,k_1,k_2)\in B^2\times [d]^2 \mid 
k_1b_1=k_2b_2\in A,~~b_1<b_2\right\}\right\rvert\leq\nonumber\\
 &\leq& \left\lvert\left\{(b_1,b_2,k_1,k_2)\in B^2\times [d]^2 \mid
k_1b_1=k_2b_2,~~b_1<b_2\right\}\right\rvert=\nonumber\\
&=& \sum_{(p,q)\in S_d}\left\lvert\left\{(b_1,b_2,k_1,k_2)\in B^2\times [d]^2 \mid 
\frac{b_1}{b_2}=\frac{k_2}{k_1}=\frac{p}{q}\right\}\right\rvert=\nonumber\\
&=& \sum_{(p,q)\in S_d}\left\lvert\left\{(b_1,b_2)\in B^2 \mid 
\frac{b_1}{b_2}=\frac{p}{q}\right\}\right\rvert\cdot
\left\lvert\left\{(k_1,k_2)\in[d]^2 \mid 
\frac{k_2}{k_1}=\frac{p}{q}\right\}\right\rvert=\nonumber\\
&=&\sum_{(p,q)\in S_d}\left\lvert\left\{(b_1,b_2)\in B^2 \mid
\frac{b_1}{b_2}=\frac{p}{q}\right\}\right\rvert\cdot
\lfloor\frac{d}{q}\rfloor=\sum_{q=2}^d\left(|\mathcal{G}_{q}(B)|-|\mathcal{G}_{q-1}(B)|\right)\lfloor\frac{d}{q}\rfloor\leq\nonumber\\
&\leq&\sum_{q=2}^d\left(|\mathcal{G}_{q}(B)|-|\mathcal{G}_{q-1}(B)|\right)\frac{d}{q}=|\mathcal{G}_d(B)|+
\sum_{q=2}^{d-1}|\mathcal{G}_q(B)|\left(\frac{d}{q}-\frac{d}{q+1}\right)=\nonumber\\
&=&|\mathcal{G}_d(B)|+\sum_{q=2}^{d-1}|\mathcal{G}_q(B)|\frac{d}{q(q+1)}.\nonumber
\end{eqnarray}

 Let $k:=\max\{\lceil1/\varepsilon\rceil,1\}$. By Proposition \ref{g}, there is a positive constant $c_2(k)$, depending
only on $k$, such that for any $c_2(k)<q\leq d$,
\begin{equation*}\label{larged}
|\mathcal{G}_q(B)|<(200k+1) n^{1+\frac{1}{k}}q^{1-\frac{1}{2k}}.
\end{equation*}
For any $q$, we have by \eqref{eq:g_easy_bound} that $|\mathcal{G}_q(B)|<n\, q^2$.

Therefore, 
\begin{eqnarray}
|W| &\leq& |\mathcal{G}_d(B)|+\sum_{q=2}^{c_2(k)}|\mathcal{G}_q(B)|\frac{d}{q(q+1)}+
\sum_{q=c_2(k)+1}^{d-1}|\mathcal{G}_q(B)|\frac{d}{q(q+1)}<\nonumber\\
&<& (200k+1)n^{1+\frac{1}{k}}d^{1-\frac{1}{2k}}+
\sum_{q=2}^{c_2(k)}n\, q^2\frac{d}{q(q+1)}+
\sum_{q=c_2(k)+1}^{d-1}(200k+1) n^{1+\frac{1}{k}}q^{1-\frac{1}{2k}}\frac{d}{q(q+1)}
\leq\nonumber\\
&\leq& (200k+1) n^{1+\varepsilon}d^{1-\frac{1}{2k}}+(c_2(k)-1) n\, d+(200k+1) 
n^{1+\varepsilon} d\sum_{q=c_2(k)+1}^{d-1}\frac{1}{q^{1+\frac{1}{2k}}}.\nonumber
\end{eqnarray}

Hence,
\begin{equation}\label{upper}
|W|< c(\varepsilon)n^{1+\varepsilon}d,
\end{equation}
where $\displaystyle c(\varepsilon):=(200k+1)+(c_2(k)-1)+
(200k+1)\sum_{q=c_2(k)+1}^{\infty}\frac{1}{q^{1+\frac{1}{2k}}}$.

\medskip

To get a lower bound for $|W|$, we define $d(a)=\lvert\{b\in B \mid \frac{a}{b}\in[d]\}\rvert$ for any $a\in A$. Then, by the convexity of the function $\binom{x}{2}=\frac{x(x-1)}{2}$:
\begin{equation}\label{lower}
|W|=\sum_{a\in A}\binom{d(a)}{2}\geq m\binom{\frac{1}{m}\sum_{a\in A}d(a)}{2}.
\end{equation}

Combining the upper and lower bounds for $|W|$, namely,
\eqref{upper} and \eqref{lower}, we get 
$$
m\binom{\frac{1}{m}\sum_{a\in A}d(a)}{2}< c(\varepsilon) n^{1+\varepsilon}d.
$$

Now, we deduce
$$
\lvert\{(a,b)\in A\times B \mid \frac{a}{b}\in[d]\}\rvert=
\sum_{a\in A}d(a)<\frac{m}{2}+\sqrt{\frac{m^2}{4}+
2c(\varepsilon) m\, n^{1+\varepsilon} d.}
$$
This implies the desired result, as $n\leq m<n\, d$.
\end{proof}

\subsection{Tightness of Proposition \ref{f}}\label{tightness}

In this section we will show that the upper bound in 
Proposition \ref{f} for $f_d(n,m)$ is essentially tight (see Proposition 
\ref{lowerg} below), provided that none of 
the parameters $m,n$, and $d$ is much larger than the product of the other two.
When one of $m,n$, and $d$ is considerably larger than the product of the 
other two, the upper bound in Proposition \ref{f} is no longer tight, 
as follows from Proposition \ref{three} below, 
in which the exact values of $f_d(m,n)$ in those cases are determined. 

\begin{propos}\label{lowerg} If $m\leq 4nd, ~n\leq 4md$, and 
$d\leq 4mn$, then $f_d(m,n)\geq\frac{1}{8}\sqrt{m\, n\, d}$.
\end{propos}

\begin{proof} Set $k=\lfloor\sqrt{m\, d/n}\rfloor,
\ell=\lfloor\sqrt{n\,d/m}\rfloor$, and 
$t=\lfloor\sqrt{m\,n/d}\rfloor$. Consider the sets
$$
A=\left\{(k+\ell)^r\,i\right\}_{r\in[t], i\in[k]},\quad \mbox{and} \quad
B=\left\{(k+\ell)^r/j\right\}_{r\in[t],j\in[\ell]}.
$$

Then $$|A|=t\, k\leq\sqrt{\frac{m\, n}{d}}\cdot\sqrt{\frac{m\,d}{n}}=m ~~~\mbox{and}$$
$$|B|=t\,\ell\leq\sqrt{\frac{m\, n}{d}}\cdot\sqrt{\frac{n\,d}{m}}=n.$$

Notice that
$$
\left\lvert\left\{(a,b)\in A\times B \mid 
\frac{a}{b}\in[d]\right\}\right\rvert\geq t\, k\,\ell \geq 
\frac{1}{2}\sqrt{\frac{m\, n}{d}}\cdot\frac{1}{2}\sqrt{\frac{m\, d}{n}}\cdot\frac{1}{2}\sqrt{\frac{n\, d}{m}}=\frac{1}{8}\sqrt{m\, n\, d}.\qedhere
$$
\end{proof}

\medskip

\begin{propos}\label{three}
\begin{enumerate}
\item If $d\geq m\, n$ then $f_d(m,n)=m\, n$.
\item If $n\geq m\, d$ then $f_d(m,n)=m\, d$.
\item If $m\geq n\, d$ then $f_d(m,n)=n\, d$.
\end{enumerate}
\end{propos}
\begin{proof}\begin{enumerate}
\item For any $A,B\subset(0,\infty)$ with $|A|\leq m,|B|\leq n$ 
we obviously have $$\lvert\{(a,b)\in A\times B \mid 
\frac{a}{b}\in[d]\}\rvert\leq|A\times B|=|A|\cdot|B|\leq m\, n.$$
To see that this upper bound can actually be attained, 
consider, for instance, the sets $A=\{1/i\}_{i\in[m]}$ and $B=[n]$.
\item For any $A,B\subset(0,\infty)$ with $|A|\leq m,|B|\leq n$ we have 
$$
\lvert\{(a,b)\in A\times B \mid \frac{a}{b}\in[d]\}\rvert=
\lvert\{(a,k)\in A\times[d] \mid \frac{a}{k}\in B\}\rvert\leq m\, d.$$
This upper bound can indeed be attained, for example by taking
$A=\{(d+1)^i\}_{i\in[m]}$ and $B=\{(d+1)^i/k\}_{i\in[m],k\in[d]}$.
\item
For any $A,B\subset(0,\infty)$ with $|A|\leq m,|B|\leq n$ we have \
$$\lvert\{(a,b)\in A\times B \mid \frac{a}{b}\in[d]\}\rvert=
\lvert\{(b,k)\in B\times[d] \mid k\cdot b\in A\}\rvert\leq n\, d.$$
Equality is attained, for example, by taking 
$A=\{(d+1)^j k\}_{j\in[n],k\in[d]}$ and $B=\{(d+1)^j\}_{j\in[n]}$.
\end{enumerate}  
\end{proof}

\section{Union of arithmetic progressions}\label{sec:prog}

In this Section we prove Theorem \ref{thm:gen-prog}.

Recall that for integers $n>1$ and $\ell>1$,
$u_{\ell}(n)$ is the minimum possible cardinality of a union of $n$ 
arithmetic progressions, each of length $\ell$, with pairwise distinct
differences. 

As a consequence of Proposition \ref{f} we get the following easy lower bound for 
$u_{\ell}(n)$ that will be useful in the regime $\ell\leq n^{\frac{1}{2}-\varepsilon}$:

\begin{propos}\label{cor:gen-prog} For any $\varepsilon>0$ there is a positive constant $c_6(\varepsilon)$, depending only on $\varepsilon$, such that for any positive integers $n$ and $\ell$  
$$u_{\ell}(n)>c_6(\varepsilon) n^{\frac{1}{2}-\varepsilon}\ell.$$ 
\end{propos}
\begin{proof} Take $n$ arithmetic progressions, each of length $\ell$, with pairwise
distinct differences, and let $U$ be their union. If each $x\in U$ belongs to less than $\sqrt{n}$ of the progressions, then 
$n\,\ell<|U|\sqrt{n}$ 
and consequently $|U|> \sqrt{n}\,\ell> n^{\frac{1}{2}-\varepsilon}\ell$.

Therefore, assume there is $x\in U$ which belongs to at least $\sqrt{n}$ progressions. In any such progression at least $d:=\lceil\frac{\ell-1}{2}\rceil$ of the terms are on the same side of $x$ (that is, either come before or after). 
Therefore, in at least $\sqrt{n}/2$ progressions there are at least $d$ terms on the same side of $x$ and without loss of generality we assume they come after
$x$ in these progressions. We now concentrate only on these progressions.
Let $B$ be the set of differences of these arithmetic progressions, 
and let $A=\{i\,b \mid i\in[d], ~~b\in B\}$. 
Proposition \ref{f} implies
\begin{align*}
d\,|B|&=\lvert\{(a,b)\in A\times B \mid \frac{a}{b}\in[d]\}\rvert\leq f_d(|A|,|B|)<\\
&<c_5(\varepsilon)\left(\min\{\sqrt{|A|},\sqrt{|B|}\}\right)^{\varepsilon}\sqrt{|A|\,|B|\,d}\leq c_5(\varepsilon)\sqrt{|B|^{\varepsilon}}\cdot\sqrt{|A|\,|B|\,d},
\end{align*}
hence
\begin{equation*}|U|\geq |\{x+a\mid a\in A\}|=|A|>\frac{1}{{c_5(\varepsilon)}^2}|B|^{1-\varepsilon}d\geq\frac{1}{{c_5(\varepsilon)}^2}\left(\frac{\sqrt{n}}{2}\right)^{1-\varepsilon}\frac{\ell-1}{2}\geq \frac{1}{{c_5(\varepsilon)}^2\,2^{3-\varepsilon}} n^{\frac{1}{2}-\varepsilon}\ell.
\end{equation*}

This completes the proof with $\displaystyle c_6(\varepsilon):=\min\{1,\frac{1}{{c_5(\varepsilon)}^2\,2^{3-\varepsilon}}\}$.
\end{proof}

The lower bounds in Theorem \ref{thm:gen-prog} in the regime 
$n^{\frac{1}{2}-\varepsilon}\leq\ell$ are established in Proposition 
\ref{prop:gen-prog} below. The proof of Proposition \ref{prop:gen-prog} uses Proposition \ref{g}, ideas similar to those appearing in the proof of Proposition \ref{f}, and the following lemma (recall the definition of $R_{d}$ from Section
\ref{sec:g}).

\begin{lem}\label{intersection} Suppose that the arithmetic progressions 
$(a_1+(j-1)b_1)_{j=1}^{\ell}$ and $(a_2+(j-1)b_2)_{j=1}^{\ell}$ have at least $r\geq 2$ common elements, then $$\frac{b_1}{b_2}\in R_{\lfloor\frac{\ell-1}{r-1}\rfloor}.$$
\end{lem}
\begin{proof}
The intersection of the arithmetic progressions 
$(a_1+(j-1)b_1)_{j=1}^{\ell}$ and $(a_2+(j-1)b_2)_{j=1}^{\ell}$ is in itself an 
arithmetic progression, whose length is at least $r\geq 2$. Suppose 
$a_1+(i_1-1)b_1=a_2+(i_2-1)b_2$ and $a_1+(j_1-1)b_1=a_2+(j_2-1)b_2$ 
are, respectively, the first and second terms of this arithmetic progression.
Since the progression is of length at least $r$, then 
$(r-1)(j_1-i_1)\leq\ell-1$ and $(r-1)(j_2-i_2)\leq\ell-1$.
It follows that  $j_1-i_1 \leq \lfloor\frac{\ell-1}{r-1}\rfloor$ and 
$j_2-i_2 \leq \lfloor\frac{\ell-1}{r-1}\rfloor$.

We also have 
$$
(j_1-i_1)b_1=(a_1+(j_1-1)b_1)-(a_1+(i_1-1)b_1)=(a_2+(j_2-1)b_2)-(a_2+(i_2-1)b_2)=
(j_2-i_2)b_2.
$$

Consequently,
\begin{equation*}
\frac{b_1}{b_2}=\frac{j_2-i_2}{j_1-i_1}\in 
R_{\lfloor\frac{\ell-1}{r-1}\rfloor}.\qedhere
\end{equation*}
\end{proof}

\begin{propos}\label{prop:gen-prog} 
For any $\varepsilon>0$ there is a positive constant $c_7(\varepsilon)$,
depending only on $\varepsilon$, 
such that for any positive integers $n$ and $\ell$ 
$$u_{\ell}(n)>c_7(\varepsilon)\min\left\{n^{1-\varepsilon}\,\ell,\,\ell^2\right\}.$$ 
\end{propos}

\begin{proof} Let $P_1,P_2,\ldots,P_n$ be $n$ arithmetic progressions,
each of length $\ell$, with pairwise distinct differences. We write
$P_i=\{a_i+(j-1)b_i\}_{j\in[\ell]}$ for $i \in [n]$, and let 
$U=\bigcup _{i=1}^n P_i=\{a_i+(j-1)b_i\mid i\in[n],j\in[\ell]\}$ be the union of
these arithmetic progressions. 
For every $x\in U$ let $\alpha(x)=|\{i\in[n] \mid x \in P_i\}|$ 
be the number of progressions containing $x$. Clearly, 
$\sum_{x\in U}\alpha(x)=n\,\ell$.
The proof will follow by comparing lower and upper bounds for the cardinality
of the set 
$$
W=\{(x,i_1,i_2) \in U\times[n]^2 \mid x\in P_{i_1}\cap P_{i_2}, ~~i_1<i_2\}.
$$

To get an upper bound on $|W|$ notice that

$$
|W|=\sum_{1\leq i_1<i_2\leq n}\left\lvert P_{i_1}\cap P_{i_2}\right\rvert=
\sum_{r=1}^{\ell-1}\left\lvert\left\{(i_1,i_2)\in[n]^2 \mid b_{i_1}<b_{i_2},
~~|P_{i_1}\cap P_{i_2}|\geq r\right\}\right\rvert.
$$
Trivially,
$$
\left\lvert\left\{(i_1,i_2)\in[n]^2 \mid b_{i_1}<b_{i_2},~~|P_{i_1}\cap P_{i_2}|\geq 1\right\}\right\rvert\leq\binom{n}{2}.
$$ 
For $r\geq 2$ we use Lemma \ref{intersection} to obtain 
$$
\left\lvert\left\{(i_1,i_2)\in[n]^2 \mid b_{i_1}<b_{i_2},~~|P_{i_1}\cap P_{i_2}|\geq
r\right\}\right\rvert\leq
\left|\mathcal{G}_{\lfloor\frac{\ell-1}{r-1}\rfloor}(B)\right|
\leq g_{\lfloor\frac{\ell-1}{r-1}\rfloor}(n).$$
Hence,
$$|W|\leq \binom{n}{2}+\sum_{r=2}^{\ell-1}g_{\lfloor\frac{\ell-1}{r-1}\rfloor}(n).$$
Let $k:=\max\{\lceil1/\varepsilon\rceil,1\}$. By Proposition \ref{g}, there is a constant $c_2(k)$ such that for any 
$2\leq r\leq\lfloor\frac{\ell-1}{c_2(k)+1}\rfloor+1$ we have
$$
g_{\lfloor\frac{\ell-1}{r-1}\rfloor}(n)<(200k+1)n^{1+\frac{1}{k}}\left(\left\lfloor\frac{\ell-1}{r-1}\right\rfloor\right)^{1-\frac{1}{2k}}\leq(200k+1)n^{1+\frac{1}{k}}(\ell-1)^{1-\frac{1}{2k}}\frac{1}{r^{1-\frac{1}{2k}}}.
$$
For $\lfloor\frac{\ell-1}{c_2(k)+1}\rfloor+2\leq r\leq\ell-1$, we use the
simpler estimate \eqref{eq:g_easy_bound} to get
$$
g_{\lfloor\frac{\ell-1}{r-1}\rfloor}(n)<n\left(\left\lfloor\frac{\ell-1}{r-1}\right\rfloor\right)^2<(c_2(k)+1)^2\,n.
$$
Therefore,
$$
|W|\leq\binom{n}{2}+\sum_{r=2}^{\lfloor\frac{\ell-1}{c_2(k)+1}\rfloor+1}(200k+1)n^{1+\frac{1}{k}}(\ell-1)^{1-\frac{1}{2k}}\frac{1}{r^{1-\frac{1}{2k}}}+\sum_{r=\lfloor\frac{\ell-1}{c_2(k)+1}\rfloor+2}^{\ell-1}(c_2(k)+1)^2\,n.
$$
Hence
\begin{equation}\label{gen:upper}|W|< c(\varepsilon)\,n\,\ell\,\max\{n/\ell,n^{1/k}\}|\leq c(\varepsilon)\, n\,\ell\,\max\{n/\ell,n^{\varepsilon}\},\end{equation}
for some positive constant $c(\varepsilon)$ depending only on $\varepsilon$.

\medskip

A simple lower bound for $|W|$ follows from the convexity
of $\binom{x}{2}=\frac{x(x-1)}{2}$:

\begin{equation}\label{gen:lower}
|W|=\sum_{x\in U}\binom{\alpha(x)}{2}\geq
|U|\binom{\frac{1}{|U|}\sum_{x\in U}\alpha(x)}{2}=
|U|\binom{n\ell/|U|}{2}.
\end{equation}

Comparing the upper and lower bounds for $|W|$, namely, 
\eqref{gen:upper} and \eqref{gen:lower}, we get 
$$
|U|\binom{n\ell/|U|}{2}< c(\varepsilon)\, n\,\ell\,\max\{n/\ell,n^{\varepsilon}\}.
$$
Hence
$$
|U|>\frac{n\,\ell}{1+2c(\varepsilon)\max\{n/\ell,n^{\varepsilon}\}},
$$
and the result follows.
\end{proof}

\bigskip

Combining Proposition \ref{cor:gen-prog} and Proposition \ref{prop:gen-prog}, we get
\medskip\newline{\bf Theorem \ref{thm:gen-prog}.} {\it For any $\varepsilon>0$ there is a positive constant $c_1(\varepsilon)$, depending only on $\varepsilon$, such that for any positive integers $n$ and $\ell$}  
\begin{equation*}\label{eq:gen-prog1}u_{\ell}(n)\geq\begin{cases} c_1(\varepsilon)\,n^{\frac{1}{2}-\varepsilon}\,\ell&\text{for }\;\ell\leq n^{\frac{1}{2}-\varepsilon}\\c_1(\varepsilon)\,\ell^2&\text{for }\; n^{\frac{1}{2}-\varepsilon}\leq\ell\leq n^{1-\varepsilon}\\ c_1(\varepsilon)\, n^{1-\varepsilon}\,\ell&\text{for }\; n^{1-\varepsilon}\leq\ell.\end{cases}\end{equation*}

\section{Further applications}\label{sec:con}
In this section we draw one (among many) possible number theoretical 
application to our upper bounds for the functions $f_d$ and $g_d$
in Sections \ref{sec:g} and \ref{sec:f}.

\begin{theorem}\label{theorem:app}
For every $\varepsilon>0$ there exists $c(\varepsilon)>0$ with the following
property.
Let $a_{1}< \ldots < a_{n}$ be $n$ natural numbers.
Then
\end{theorem}
\begin{equation}\label{eq:app}
\sum_{1\leq i<j\leq n}\frac{\text{gcd}(a_{i},a_{j})}{a_{j}}
< c(\varepsilon)n^{1+\varepsilon}.
\end{equation}

\begin{proof} Denote $B=\{a_{1}, \ldots, a_{n}\}$. Notice that every summand on the left hand side of (\ref{eq:app}) is
of the form $\frac{1}{k}$ for some positive integer $k$.
The simple but crucial observation is that if $1\leq i<j\leq n$ such that
$\frac{\text{gcd}(a_{i},a_{j})}{a_{j}}=\frac{1}{k}$, then 
$\frac{a_{i}}{a_{j}} \in R_{k}$. Therefore, $\frac{\text{gcd}(a_{i},a_{j})}{a_{j}}=\frac{1}{d}$, for $1\leq i<j\leq n$, if and only if
$\{a_i,a_j\} \in \mathcal{G}_{d}(B) \setminus \mathcal{G}_{d-1}(B)$.
(Recall the definition of $R_{k}$ and $\mathcal{G}_{d}(B)$ in Section \ref{sec:g}.)

Fix a positive integer $k$, to be determined later. 
By Proposition \ref{g}, there exists $c_2(k)>0$ such that for every $d>c_2(k)$
\begin{equation}\label{eq:T}
|\mathcal{G}_{d}(B)| \leq g_d(n)<(200k+1)n^{1+\frac{1}{k}}d^{1-\frac{1}{2k}}.
\end{equation}

For every $d$, $|\mathcal{G}_{d}(B)| \leq g_d(n)<nd^2$, by \eqref{eq:g_easy_bound}.
This easy upper bound will be useful when $d$ is small (smaller than $c_2(k)$).

We are now ready to prove the Theorem.

\begin{eqnarray}
\sum_{1\leq i<j\leq n}\frac{\text{gcd}(a_{i},a_{j})}{a_{j}} &=& 
\sum_{d\geq 2}\left\lvert\left\{(i,j)\in[n]^2\mid i<j,\,\frac{\text{gcd}(a_{i},a_{j})}{a_{j}}=\frac{1}{d}\right\}\right\rvert\cdot\frac{1}{d}=\nonumber\\
&=&\sum_{d \geq 2}\left(|\mathcal{G}_{d}(B)|-|\mathcal{G}_{d-1}(B)|\right)\frac{1}{d}=
\sum_{d \geq 2}|\mathcal{G}_{d}(B)|\left(\frac{1}{d}-\frac{1}{d+1}\right)=\nonumber\\
& = & \sum_{2 \leq d \leq c_2(k)}\frac{1}{d(d+1)}|\mathcal{G}_{d}(B)|+
\sum_{d > c_2(k)}\frac{1}{d(d+1)}|\mathcal{G}_{d}(B)| <\nonumber\\
&< &\sum_{2 \leq d \leq c_2(k)}\frac{1}{d(d+1)}n\, d^2+
\sum_{d > c_2(k)}\frac{1}{d(d+1)}(200k+1)n^{1+\frac{1}{k}}d^{1-\frac{1}{2k}}\leq\nonumber\\
&\leq&c_2(k)n+(200k+1)n^{1+\frac{1}{k}}\sum_{d > c_2(k)}\frac{1}{d^{1/2k}(d+1)}.
\nonumber
\end{eqnarray}

Take $k$ to be a positive integer such that $\frac{1}{k}< \varepsilon$
and let $c(\varepsilon)=
c_2(k)+(200k+1)\sum_{d > c_2(k)}\frac{1}{d^{1/2k}(d+1)}$
to get the desired result.
\end{proof}

\begin{urem} It is not hard to verify that the bound in 
Theorem \ref{theorem:app} cannot be improved to be linear in $n$.
This can be seen for example by taking $a_{1}, \ldots, a_{n}$ to be
$1, \ldots, n$, respectively. Then a direct computation, using some classical 
number theory estimates, show that in this case the left hand side of 
(\ref{eq:app}) is $\Theta(n\log n)$.
\end{urem}

Theorem \ref{theorem:app} allows us to write in a slightly different way the
proof of Proposition \ref{prop:gen-prog}, giving the lower bound for 
$u_{\ell}(n)$.

Indeed, suppose we wish to bound from below the union of 
$n$ arithmetic progressions, $A_{1}, \ldots, A_{n}$, each of length $\ell$,
with pairwise distinct differences $a_{1}, \ldots, a_{n}$, respectively. With no loss of generality we may assume that $a_{1}< \ldots< a_{n}$ and that they are all positive integers.
We will use the following 
well known estimate of Dawson and Sankoff (\cite{DawsonSankoff}) on the cardinality of the union of sets
via the cardinalities of their pairwise intersections.

\begin{equation}\label{eq:u-i}
|\bigcup_{i=1}^{n}A_{i}| \geq 
\frac{(\sum_{i=1}^{n}|A_{i}|)^2}{\sum_{1 \leq i,j \leq n}|A_{i} \cap A_{j}|}.
\end{equation}

Hence, we examine the cardinalities of the pairwise intersections of the
progressions. 

Consider two progressions of length $\ell$: $\{p+(j-1)q\}_{j=1}^{\ell}$ and
$\{p'+(j-1)q'\}_{j=1}^{\ell}$, where $q,q'$ are positive integers. Their intersection is in itself an arithmetic 
progression and it is not hard to see that the difference of this
progression (assuming it has at least two elements) 
is equal to the smallest number divisible by both $q$ and $q'$.
It follows that the size of the intersection of the two progressions
is less than or equal to $1+\frac{\min(\ell q, \ell q')}{\text{lcm}(q,q')}=
1+\ell\frac{\text{gcd}(q,q')}{\max(q,q')}$.

It follows from the above discussion that the union $|\bigcup_{i=1}^{n}A_{i}|$
is bounded from below by 
$$
\frac{(n\ell)^2}{n\ell+n^2+2\ell\sum_{1\leq i<j\leq n}\frac{\text{gcd}(a_{i},a_{j})}{a_{j}}}.
$$

In view of Theorem \ref{theorem:app}, this expression is greater than
$\min(\frac{1}{3c(\varepsilon)}n^{1-\varepsilon}\ell, \frac{1}{2}\ell^2)$.

\bigskip

It is interesting to note the relation of Theorem \ref{theorem:app}
to a well known conjecture of Graham (\cite{Graham}).
Graham conjectured that given any $n$ positive integers $a_{1}< \ldots <a_{n}$,
there are two of them that satisfy $\frac{a_{j}}{\text{gcd}(a_{i},a_{j})} \geq n$.
This conjecture has a long history with many contributions. It was finally
completely (that is, for all values of $n$) solved in \cite{BalasubramanianSoundararajan}, 
where one can also find more details on the history and references related
to this conjecture.
 
From \eqref{eq:app} it follows that there is a pair of indices
$1\leq i<j\leq n$ such that 
$\frac{\text{gcd}(a_{i},a_{j})}{a_{j}} < \frac{c(\varepsilon)n^{1+\varepsilon}}{\binom{n}{2}}.$ 
This implies $\frac{a_{j}}{\text{gcd}(a_{i},a_{j})} > \frac{1}{2c(\varepsilon)}n^{1-\varepsilon}$. This lower bound is indeed much weaker than the desired one in the conjecture of Graham, but on the 
other hand this argument shows that ``in average'' 
$\frac{a_{j}}{\text{gcd}(a_{i},a_{j})}$ is 
quite large.

\bigskip

{\bf \Large Acknowledgments}

\medskip

We thank Seva Lev for interesting discussions about the problem and for 
pointing out the relation of Theorem \ref{theorem:app} to the conjecture
of Graham.

\end{document}